\documentclass[12pt,reqno]{amsart}
\usepackage{amssymb,latexsym}
\usepackage{color}
\usepackage{graphicx}
\usepackage[mathcal]{euscript}
\usepackage{mathrsfs}
\usepackage{enumerate}
\usepackage{hyperref}

\newtheorem{thm}{Theorem}[section]
\newtheorem{cor}[thm]{Corollary}
\newtheorem{lem}[thm]{Lemma}

\newtheorem{prop}[thm]{Proposition}

\theoremstyle{definition}
\newtheorem{defn}[thm]{Definition}
\newtheorem{rem}[thm]{Remark}
\newtheorem{exmp}[thm]{Example}

\numberwithin{equation}{section}
\allowdisplaybreaks

\begin{document}

\title{Slant H-Toeplitz Operators on the Hardy space}

\author[\textbf{A. Gupta} and \textbf{S. K. Singh}]{\textbf{Anuradha Gupta} and \textbf{Shivam Kumar Singh}}
\address{Anuradha Gupta, Department of Mathematics, Delhi college of Arts and Commerce , University of Delhi, Delhi, India}\email{dishna2@yahoo.in}
\address{Shivam kumar Singh, Department of Mathematics, University of Delhi, New Delhi 110007, India}\email{shivamkumarsingh14@gmail.com}

\begin{abstract}
The notion of slant H-Toeplitz operator $V_\phi$ on the Hardy space $H^2$ is introduced and its  characterizations are obtained. We have shown that an operator on the space $H^2$ is slant H-Toeplitz if and only if its matrix is a slant H-Toeplitz matrix. In addition the conditions under which  slant Toeplitz and slant Hankel operators become slant H-Toeplitz operators are also obtained. 
\end{abstract}

 \maketitle


\section{Introduction}
\noindent Let $\mu$ denote the normalised Lebesgue measure on the unit circle $\mathbb{T}$ and the space $L^2$ be  the space of all complex valued square integrable measurable functions on $\mathbb{T}$. The space $L^2$ is  a Hilbert space with the norm $\|\cdot\|_2$ induced by the inner product 
\[\left\langle f, g \right\rangle =  \int f\bar{g} d\mu \quad \text{ for }f,g\in L^2.\]
For each integer $n$, let $e_n(z)=z^n$ for $z \in \mathbb{T}$. Then the collection $\left\{e_n\right\}_{n \in \mathbb{Z}}$ forms an orthonormal basis for $L^2$, where $\mathbb{Z}$ denote the set of integers. 
The Hardy space $H^2= \{f\in L^2: f \text{ is analytic}\}$  being a closed subspace of $L^2$ is a Hilbert space under the norm 
\[\|f\|= \left(\sum_{n=0}^\infty|a_n|^2\right)^{1/2} < \infty \ \text{ where } f(z)= \sum_{n=0}^\infty a_nz^n.\]
The space $L^\infty$ denotes the Banach space of all essentially bounded measurable functions with norm given by $\|\phi\|_\infty = \text{ess sup}\left\{|\phi(z)|: z \in \mathbb{T}\right\}$.
Let $\mathcal{B}(L^2)$ and $\mathcal{B}(H^2)$ denote the set of all bounded linear operator on the spaces $L^2$ and $H^2$ respectively.
 Let $P$ denote the orthogonal projection from the space $L^2$ to the space $H^2$. For a given $\phi \in L^\infty$, the induced  multiplication operator $M_\phi: L^2 \longrightarrow L^2$ is defined as $M_\phi f= \phi f$ for each $f\in L^2$ and the Toeplitz operator is  the operator $T_\phi \in \mathcal{B}(H^2)$ such that $T_\phi= PM_{\phi}{|_{H^2}}$. For the symbol $\phi \in L^\infty$, Hankel operator $H_\phi \in \mathcal{B}(H^2)$ is defined as the operator $H_\phi = PM_\phi J$, where $J$, the flip operator, is the operator  $J: H^2 \longrightarrow (H^2)^\perp $  given by $J(e_n)= e_{-n-1}\text{ for } n\geq 0$.
The slant Toeplitz operator \cite{ho1, ho2, datt} with the symbol $\phi$ is defined as the operator $A_\phi \in \mathcal{B}(L^2)$ such that $A_\phi = WM_\phi$, where the operator $W$ defined on $L^2$ is given by $$W e_{n}= \begin{cases}
  e_{\frac{n}{2}} & \text{if $n$ is even} \\
  0 & \text{otherwise}
\end{cases}$$ for each integer $n$ and its adjoint is given by $W^\star e_{n}=  e_{2n}$. The compression of slant Toeplitz operator \cite{arora5} to the space $H^2$ is the operator $B_\phi$ defined by $B_\phi= P A_{\phi}{|_{H^2}}$. The slant Hankel operator \cite{arora4,dattR} on the space $H^2$ is given by $L_\phi$ such that $L_\phi= WH_\phi$.
For a non-constant analytic function $\phi$, the composition operator \cite{cowen83, clif99} is  the operator $C_\phi $ defined on $H^2$ such that $C_\phi(f)(z)=  f(\phi(z))$ for each $f \in H^2$ and $z\in \mathbb{T}$.

The study  of slant Toeplitz operators has gained voluminous importance due to its multidirectional applications as these classes of operators have played major role in wavelet analysis, dynamical system and in curve and surface modelling (\cite{goodman1994spectral, heil1996approximation,  strang1994orthogonal, strang1995short, villemoes1994wavelet}). The study of Hankel and  slant hankel operators has numerous applications, in interpolation problems, Hamburger's moment problem, rational approximation theory and stationary process.
 In 2007, Arora et al. \cite{arora7} introduced and studied the notion of H-Toeplitz operators on the space $H^2$. Motivated by these studies, we have introduced the notion of slant H-Toeplitz operators on Hardy space $H^2$  and studied their properties. We have obtained the conditions under which the slant H-Toeplitz operators are self adjoint, compact and hyponormal. Moreover, we have also obtained the characterizations for slant H-Toeplitz operators. In particular we have shown that an operator on $H^2$ is a slant H-Toeplitz operator if and only if its matrix is a slant H-Toeplitz matrix.    

\section{Slant H-Toeplitz Operators}
\noindent The H-Toeplitz operator \cite{arora7} with a symbol $\phi$ is the operator $S_\phi \in  \mathcal{B}(H^2) \text{ defined by } S_\phi(f)= PM_\phi K(f) \ \text{for all }f \in H^2$, 
 where the operator $K : H^2 \longrightarrow L^2$ is given by $K(e_{2n})= e_n$ and $K(e_{2n+1})= e_{-n-1}$ for all non-negative integers $n$. The adjoint $K^*$ of the operator $K$  is given by $K^*(e_n)= e_{2n}, \ K^*(e_{-n-1})= e_{2n+1}$ for $n\geq 0$. Thus, $K^*K = I \text{ on } H^2$ and $KK^*=I \text { on }L^2$.
\noindent Motivated by the definition H-Toeplitz operator,  we define slant H-Toeplitz operator on the space $H^2$ as follows:
\begin{defn} For $\phi \in L^\infty$, the slant H-Toeplitz operator is defined as the operator $V_\phi :  H^2 \longrightarrow H^2$ such that $V_\phi(f)= WPM_\phi K(f)$ for each $f$ in $H^2$.
\end{defn}
\noindent The operator $V_\phi$ with symbol $\phi \in L^\infty$, is a bounded linear operator as we have $
  \|V_\phi\| = \|WS_\phi\|= \|WPM_\phi K \|  \leq \|W\| \|\phi\|_\infty \|K\|  \leq \|\phi\|_\infty.$
\begin{thm} The correspondence $\phi \ \longrightarrow \ V_\phi$ is one-one.
\end{thm}
\begin{proof}
Let $ \phi(z)= \sum_{n=-\infty}^\infty a_nz^n$, $ \psi(z)= \sum_{n=-\infty}^\infty b_nz^n \in L^\infty $ be such that $V_\phi = V_\psi$. Therefore, $V_\phi- V_\psi=0$ or, equivalently, $WPM_{\phi-\psi}K=0$. That implies that 
\begin{equation}\label{eqns}
WPM_{\phi-\psi}K(e_m)= 0 \ \text{for} \ m\geq 0.
\end{equation} Therefore, in particular, we have 
$ WPM_{\phi-\psi}K(e_{2m}(z))= 0$ which gives $WP \sum _{n=-\infty}^\infty \left(a_{n}-b_{n} \right)z^{n+m} =0 $, or, $\sum _{n=0}^\infty \left(a_{2n-m}-b_{2n-m} \right)z^{n} =0 $.
 Therefore, 
$\left\langle WPM_{\phi-\psi}K(e_{2m}(z)), WPM_{\phi-\psi}K(e_{2m}(z)) \right\rangle = 0$ which implies that $\left\langle \sum _{n=0}^\infty \left(a_{2n-m}-b_{2n-m} \right)z^n , \sum _{n=0}^\infty \left(a_{2n-m}-b_{2n-m} \right)z^n  \right\rangle = 0$,
or equivalently, $\sum_{n=0}^\infty \left|a_{2n-m}- b_{2n-m}\right|^2= 0.$
Thus, it follows that $a_{2n-m}= b_{2n-m} \ \text{for } \ n,m\geq 0$. Similarly, using equation \eqref{eqns}, we get that $WPM_{\phi-\psi}K(e_{2m+1}(z))= 0$
which on using the definitions of operators $W,P$ and $K$ shows that $\sum _{n=0}^\infty \left(a_{2n+m+1}-b_{2n+m+1} \right)z^n =0$ and therefore it follows that $a_{2n+m+1}= b_{2n+m+1}$ for each $n\geq 0$ and $m\geq 0$. Hence, $a_n= b_n$ for all integers $n$ and this proves $\phi= \psi$.
\end{proof}

\noindent Let $\phi = \sum_{{n= -\infty}}^\infty a_n e_n \in L^\infty$ and $(a_{i,j})$ be the matrix of slant H-Toeplitz operator $V_\phi$ with respect to the orthonormal basis $\{e_n \}_{n\geq 0}$, where the $(i,j)^{th}$ entry, $a_{i,j}= \left\langle V_\phi e_j, e_i \right\rangle$ satisfies the following:
\begin{flalign*}
&\quad a_{k,0}= \begin{cases}
a_{k+j,4j}, & \text{for all } j\geq 0 \text{ and } k\geq 0 \\ 
a_{k-j, 4j-1}, & \text{for all } j=1,2,3,\dots , k-j \geq0 \text{ and } k\geq 1
\end{cases}& \\
\text{and }&\quad a_{0,2k}= a_{i, 2k+4i}\quad \text{ for all } i \geq 1, k\geq 1.
\end{flalign*}
Therefore, the matrix of slant H-Toeplitz operator explicitly is given by
\[V_\phi= \begin{bmatrix}\label{mainmatrix}
a_0 & a_1 & a_{-1} & a_2 & a_{-2} & a_{3} & a_{-3}\cdots\\
a_2 & a_3 & a_{1} & a_4 & a_{0} & a_{5} & a_{-1}\cdots\\
a_4 & a_5 & a_{3} & a_6 & a_{2} & a_{7} & a_{1}\cdots\\
a_6 & a_7 & a_{5} & a_8 & a_{4} & a_{9} & a_{3}\cdots\\
a_8 & a_9 & a_{7} & a_{10} & a_{6} & a_{11} & a_{5}\cdots\\
\vdots & \vdots & \vdots & \vdots & \vdots & \vdots & \vdots
\end{bmatrix}\]
which is a two way infinite matrix and it is an upper triangular matrix in the case when the symbol $\phi$ is co-analytic.
Also for each non-negative integer $n$, it follows  that\[ V_\phi(e_{2n})= WPM_\phi K(e_{2n})= WPM_\phi(e_n)= WT_\phi(e_n)= B_\phi(e_n)\]
and \begin{align*}
V_\phi(e_{2n+1})&= WPM_\phi K(e_{2n+1})= WPM_\phi (e_{-n-1}) = WPM_\phi J(e_n)\\
&=WH_\phi(e_n)= L_\phi(e_n).
\end{align*}
This shows that the matrix of slant Toeplitz operator $B_\phi$ can be obtained by deleting every odd column of the matrix of slant H-Toeplitz operator $V_\phi$ and the matrix of slant Hankel operator $L_\phi$ can be obtained by deleting every even column of the matrix  $V_\phi$. Hence, the $(i,j)^{th}$ entry of the matrix of $V_\phi$ is given by :
\[a_{i,j}= \begin{cases}
a_{2i} & \text{if } j=0 \\
a_{2i-n} & \text{if } j = 2n\\
a_{2i+n+1}& \text{if } j = 2n+1
\end{cases}\]
where $n \in \mathbb{N}$. This motivates us to define the slant H-Toeplitz matrix in the following way:

\begin{defn} A two way doubly infinite matrix $(a_{i,j})$ is said to be a slant H-Toeplitz matrix if it satisfies the following
\begin{equation}\label{matrix}
a_{k,0}= \begin{cases}
a_{k+j,4j}, & \text{for all } j\geq 0, \text { and } k\geq 0 \\ 
a_{k-j, 4j-1}, & \text{for all } j=1,2,3,\dots , k-j \geq0, \text{ and } k\geq 1
\end{cases}
\end{equation}
and
\begin{equation}\label{mmatrix}
 a_{0,2k}= a_{i, 2k+4i}, \quad \text{for all } i \geq 1\text{ and } k\geq 1.
\end{equation}
\end{defn}
The H-Toeplitz matrix give rises to slant Toeplitz and slant Hankel operators that can be seen by the following theorem.
\begin{thm}\label{A1} If a matrix of any bounded linear operator $A$ defined on $H^2$ is a slant H-Toeplitz matrix, then $AC_{z^2}$ is a slant Toeplitz operator and $AM_zC_{z^2}$ is a slant Hankel operator. 
\end{thm}
\begin{proof}
Let $A$ be a bounded linear operator on $H^2$ such that its matrix $(a_{i,j})$ with respect to the orthonormal basis $\{e_n \}_{n\geq 0}$ is a slant H-Toeplitz matrix and therefore, it satisfies relations (\ref{matrix}),\eqref{mmatrix}. 
Let $(\alpha_{i,j})$ be the matrix of bounded linear operator $AC_{z^2}$, defined on $H^2$ with respect to the orthonormal basis $\{e_n\}_{n\geq 0}$. Then using the definition of slant H-Toeplitz matrix, we have 
\begin{flalign*}
\alpha_{i+1,j+2} &= \left\langle AC_{z^2} z^{j+2}, z^{i+1} \right\rangle = \left\langle A z^{2j+4}, z^{i+1} \right\rangle  
= a_{i+1, 2j+4} = a_{i,2j} &\\
&= \left\langle A z^{2j}, z^i \right\rangle = \left\langle AC_{z^2} z^j, z^i\right\rangle = \alpha_{i,j}  \text{ for all } i,j \geq 0.
\end{flalign*}
Therefore, $(\alpha_{i,j})$ is a slant Toeplitz matrix and hence the operator $AC_{z^2}$ is a slant Toeplitz operator.
Next, let $(\beta_{i,j})$ be the matrix of the bounded linear operator $AM_zC_{z^2}$, defined on $H^2$ with respect to the basis $\{e_n \}_{n\geq 0}$. Then, by the definition of slant H-Toeplitz matrix, it follows that
\begin{flalign*}
\beta_{i-1, j+2}&= \left\langle AM_zC_{z^2} z^{j+2}, z^{i-1} \right\rangle = \left\langle A z^{2j+5}, z^{i-1} \right\rangle 
= a_{i-1,2j+5}= a_{i,2j+1} &\\
&= \left\langle Az^{2j+1}, z^i \right\rangle = \left\langle AM_zC_{z^2} z^j, z^i \right\rangle 
= \beta_{i,j}  \text { for all } i\geq 1, j \geq 0.
\end{flalign*}
Thus, the matrix ($\beta_{i,j}$) is a slant Hankel matrix and hence the operator $AM_zC_{z^2}$ is a slant Hankel operator.
\end{proof}
\begin{cor} If the matrix of a bounded linear operator $A$ defined on $H^2$ is the slant H-Toeplitz matrix, then 
 $AC_{z^2}= WT_\phi$ and $AM_zC_{z^2}= WH_\phi$ for some $\phi \in L^\infty$.
\end{cor}
\begin{proof}
Let $A \in \mathcal{B}(H^2)$ be such that its matrix $(a_{i,j})$ with respect to orthonormal basis $\{e_n \}_{n\geq 0}$ is a slant H-Toeplitz matrix. Therefore, by Theorem \ref{A1}, the operators $AC_{z^2}$ and $AM_zC_{z^2}$ are slant Toeplitz operator and slant Hankel operator, respectively. Let $(\alpha_{i,j}) \text{ and } (\beta_{i,j})$ be the matrices of the operators $AC_{z^2} \text{ and } AM_zC_{z^2}$, respectively, with respect to standard basis $\{e_n \}_{n\geq 0}$. Then by the definition of slant H-Toeplitz matrix, it follows that
\[\alpha _{k,0}= \left\langle WT_\phi z^0, z^k \right\rangle = \left\langle AC_{z^2}z^0, z^k\right\rangle = \left\langle Az^0, z^k\right\rangle = a_{k,0}\]
and \[\alpha_{0,j}= \left\langle WT_\phi z^j, z^0 \right\rangle = \left\langle AC_{z^2}z^j, z^0 \right\rangle= \left\langle A z^{2j}, z^0\right\rangle = a_{0,2j}.\]
Thus, for all $k \geq 0, j \geq 1$, it follows that $\alpha_{k,0}= \left\langle \phi, z^{2k} \right\rangle = a_{k,0}$  and $\alpha_{0,j}= \left\langle \phi, z^{2j}\right\rangle = a_{0,2j}$.
Also by equation \eqref{matrix}, $\alpha_{i,j}= \left\langle AC_{z^2} z^j, z^i \right\rangle= \left\langle Az^{2j}, z^i \right\rangle= a_{i,2j}= a_{2i-j,0}= \alpha_{2i-j}.$ Now since $AM_zC_{z^2}= WH_{{\phi}}$, therefore, for $k\geq 0$ and by equation \eqref{matrix} it follows that
\[\beta_{k,0}= \left\langle WH_{{\phi}}z^0, z^k \right\rangle= \left\langle AM_zC_{z^2}z^0,z^k \right\rangle= \left\langle Az, z^k\right\rangle= a_{k,1}= a_{k+1,2}.\]
Also, $\left\langle WH_{{\phi}} z^0,z^k\right\rangle= \left\langle PM_{{\phi}}J z^0, z^{2k} \right\rangle= \left\langle M_{{\phi}} z^{-1}, z^{2k} \right\rangle= \left\langle {\phi}, z^{2k+1} \right\rangle$ and so $\left\langle {\phi}, z^{2k+1} \right\rangle$= $a_{k,1}$. Since, $\left\langle \phi, z^{2k} \right\rangle = a_{k,0}$ and $\left\langle {\phi}, z^{2k+1} \right\rangle$= $a_{k,1}$, therefore we define the function $\phi$ as follows  
\begin{equation}\label{phi}
\left\langle \phi, z^{k} \right\rangle = \begin{cases}
a_{k/2,0}, &  k\geq 0 \text{ and $k$ is even} \\
a_{(k-1)/2,1} & k\geq 0 \text{ and $k$ is odd}\\
a_{0,-2k}, &  k\leq -1. 
\end{cases}
\end{equation}
Hence, the operator $AC_{z^2}$ is a slant Toeplitz operator $B_\phi$ and the operator $AM_zC_{z^2}$ is a slant Hankel operator $S_\phi$ with $\phi$ defined by \eqref{phi}. Also, $\beta_{i,j}= \left\langle AM_zC_{z^2} z^j, z^i\right\rangle= \left\langle Az^{2j+1}, z^i\right\rangle= a_{i,2j+1}= a_{i-1, 2j+5}= \left\langle AM_zC_{z^2} z^{j+2}, z^{i-1}\right\rangle= \beta_{i-1,j+2}$ for $i\geq 1, j\geq 0$. 
\end{proof}

\begin{rem}
From the matrix representation, given by the relations \eqref{matrix} and \eqref{mmatrix} of slant H-Toeplitz operator $V_\phi$ with symbol $\phi \in L^\infty$, it follows that  with respect to a suitable basis on the domain and range spaces for the operator $V_\phi$, the matrix of  $V_\phi$ can be represented as  the matrix whose columns on the left side are of the matrix of $B_\phi$ and the columns on the right side are of the matrix of $L_\phi$.  
Therefore, with respect to above representation, we can conclude that any slant H-Toeplitz operator is unitarily equivalent to a direct sum of a slant Toeplitz operator and a slant Hankel operator.
\end{rem}

 It can be observed that, for $\phi \in L^\infty$, the adjoint of the operator $V_\phi$ is the operator $V_\phi^*$ satisfying $$V_\phi^* = \left(WPM_\phi K\right)^*= K^*M_\phi^*P^*W^*=K^* M_{\bar{\phi}}W^*.$$
Then, for $i,j \geq 0$ and for $\phi = \sum _{n= - \infty}^\infty a_ne_n \in L^\infty$, we have
\begin{flalign*}
\left\langle V_\phi^* e_j, e_i \right\rangle &= \left\langle K^* M_{\bar{\phi}}PW^*e_j, e_i \right\rangle = \left\langle K^*M_{\bar{\phi}}e_{2j}, e_i \right\rangle &\\
&= \Big\langle K^*\Big(\sum_{n= -\infty}^\infty \overline{a_n}e_{2j-n}, e_i \Big)\Big\rangle =\sum_{n= -\infty}^\infty \overline{a_n}\left\langle e_{2j-n}, e_i\right\rangle &\\
&= \overline{a_{2j-i}}
\end{flalign*}
 and hence the matrix of $V_\phi^*$ with respect to orthonormal basis $\{e_n\}_{n\geq 0}$ is given by
\[V_\phi^*= \begin{bmatrix}
\overline{a_0} & \overline{a_2} & \overline{a_{4}} & \overline{a_6} & \overline{a_{8}} & \overline{a_{10}} & \overline{a_{12}}\cdots\\
\overline{a_{-1}} & \overline{a_1} & \overline{a_{3}} & \overline{a_5} & \overline{a_{7}} & \overline{a_{9}} & \overline{a_{11}}\cdots\\
\overline{a_{-2}} & \overline{a_0} & \overline{a_{2}} & \overline{a_4} & \overline{a_{6}} & \overline{a_{8}} & \overline{a_{10}}\cdots\\
\overline{a_{-3}} & \overline{a_{-1}} & \overline{a_{1}} & \overline{a_3} & \overline{a_{5}} & \overline{a_{7}} & \overline{a_{9}}\cdots\\
\overline{a_{-4}} & \overline{a_{-2}} & \overline{a_{0}} & \overline{a_2} & \overline{a_{4}} & \overline{a_{6}} & \overline{a_{8}}\cdots\\
\vdots & \vdots & \vdots & \vdots & \vdots & \vdots & \vdots
\end{bmatrix}.\]
Also, we have $\|V_\phi^* f\|^2 = \|(WT_\phi)^*f\|^2 \ + \ \|(WH_\phi)^*f\|^2$ for each $f \in H^2$. 
\begin{prop} If $\phi \in L^\infty$ is an inner function, then the operator $V_\phi^*$ is an isometry. 
\end{prop}
\begin{proof} If $\phi\in L^\infty$ is an inner function, then $|\phi|=1$ a.e. on $\mathbb{T}$. Then, for each non-negative integer $n$ and by the definitions of operators $ V_{\phi}$ and $ V_{\phi}^*$ it follows that
$V_\phi V_\phi^* (e_n) = \left(WPM_\phi K\right)\left(K^*M_{\bar{\phi}}PW^*\right)(e_n) 
= WPM_\phi (KK^*)M_{\bar{\phi}}(e_{2n})
= WPM_{|\phi|^2}(e_{2n}) = WW^*(e_n)= (e_n)$.
Thus, $V_\phi V_\phi^* (e_n)= (e_n)$ for all $n\geq 0$. Hence, the operator  $V_\phi^*$ is an isometry.
\end{proof}
\noindent The condition in the above theorem is only necessary but not sufficient as shown by the following example.
\begin{exmp} For $\phi(z)= ({1+z})/{\sqrt{2}}$, we have 
\begin{flalign*}
V_\phi V_\phi^* (e_n(z)) &= WP M_\phi M_{\bar{\phi}}(z^n) = \frac{1}{\sqrt{2}} W P \left(\frac{1+z}{\sqrt{2}} \left(z^{2n}+ z^{2n-1} \right)  \right)&\\
& = \frac{1}{2} WP \left(z^{2n}+z^{2n-1}+z^{2n+1}+ z^{2n} \right) = z^n.
\end{flalign*}
Therefore, $V_\phi V_\phi^* (e_n) = e_n$ for each non-negative integer $n$ and hence $V_\phi^*$ is an isometry, but  $\phi$ is not an inner function. 
\end{exmp}
\begin{thm} Let $\phi = \sum\limits _{n=-\infty}^\infty a_n e_n \in L^\infty$ and the operator $V_\phi^*$ is an isometry on $H^2$. Then  $\sum\limits _{n=-\infty}^\infty |a_n|^2 =1$.
\end{thm}
\begin{proof}
If the operator $V_\phi^*$ is an isometry, then $ V_\phi V_\phi^* =I$ which from the definition of the operator $V_\phi$ implies that $WT_{|\phi|^2}W^* = WW^*$ or equivalently,  $W\left(I-T_{|\phi|^2}\right)W^* =0$ 
 \text{ and therefore we get that } $WT_{1-|\phi|^2}W^* = 0.$ Thus, for $m\geq 0$, it follows that 
$\left\langle WT_{1-|\phi|^2}W^* e_m, e_m \right\rangle = 0$,  
 that is, $\left\langle (T_{1-|\phi|^2}) e_{2m}, e_{2m} \right\rangle =0$\text{ which gives } $\left\langle (1-|\phi|^2)e_{2m}, e_{2m} \right\rangle =0$
 and so $\left\langle z^{2m}, z^{2m}\right\rangle - \left\langle \phi(z) \overline{\phi(z)}z^{2m}, z^{2m}\right\rangle =0.$ Therefore, on substituting the value of $\phi$, we get that $\left\langle \sum_{n=-\infty}^\infty a_n z^{n+2m}, \sum_{k=-\infty}^\infty a_k z^{k+2m}  \right\rangle =1$ or, equivalently, we have $\sum_{n=-\infty}^\infty a_n  \sum_{k=-\infty}^\infty \overline{a_k} \left\langle z^{n+2m}, z^{k+2m} \right\rangle =1 $. Hence, it follows that $ \sum_{n=-\infty}^\infty |a_n|^2 =1$.
\end{proof}

\noindent If $\phi \in L^\infty$ is an inner function, then the operator $V_\phi$ being coisometry is also a partial isometry and the converse follows from the following theorem.
\begin{thm}If $\phi \in L^\infty$ and the operator $V_\phi$ is partial isometry, then \[\left(WT_{1-|\phi|^2}W^*\right)WT_\phi K =0.\]
\end{thm}
\begin{proof}
Let the operator $V_\phi$ be a partial isometry on $H^2$. Therefore, $
 \quad V_\phi = V_\phi V_\phi^* V_\phi,   
\text{ that is, } V_\phi= \left( WT_{|\phi|^2}W^* \right)V_\phi$
\text{ which further implies that } $\left( I- WT_{|\phi|^2}W^* \right)V_\phi =0$. \text{ So, } $\left( WPM_1W^*- WT_{|\phi|^2}W^* \right)WPM_\phi K = 0, 
 \text{ or equivalently, } \left(WT_{1-|\phi|^2} W^* \right)WPM_\phi K=0$ 
 and hence it follows that $\left(WT_{1-|\phi|^2}W^*\right)WT_\phi K =0.$
\end{proof}
\begin{thm}
 For $\phi \in L^\infty$, the operator $V_\phi$ is a Hilbert-Schmidt operator if and only if $\phi=0$.
\end{thm}
\begin{proof}
Clearly if $\phi=0$, then the operator $V_\phi$ is a Hilbert-Schmidt operator. Let $\phi= \sum\limits_{n=-\infty}^\infty a_ne_n \in L^\infty$ and the operator $V_\phi$ be a Hilbert-Schmidt operator.  From the definitions of the operators $W$ and $K$, we see that
\begin{flalign*}
&\sum_{m=0}^\infty  \left\langle V_\phi e_m, V_\phi e_m\right\rangle = \sum_{m=0}^\infty  \left\langle V_\phi e_{2m}, V_\phi e_{2m}\right\rangle + \sum_{m=0}^\infty  \left\langle V_\phi e_{2m+1}, V_\phi e_{2m+1}\right\rangle &\\
&= \sum_{m=0}^\infty  \left\langle WPM_\phi e_m, WPM_\phi e_m\right\rangle + \sum_{m=0}^\infty  \left\langle WPM_\phi e_{-m-1}, WPM_\phi e_{-m-1}\right\rangle &\\
&= \sum_{m=0}^\infty  \Big\langle WP \sum_{n=-\infty}^\infty a_ne_{n+m}, WP \sum_{n=-\infty}^\infty a_ne_{n+m}\Big\rangle &\\
& \qquad +\sum_{m=0}^\infty  \Big\langle WP \sum_{n=-\infty}^\infty a_ne_{n-m-1}, WP \sum_{n=-\infty}^\infty a_ne_{n-m-1}\Big\rangle &\\
&= \sum_{m=0}^\infty  \Big\langle \sum_{n=0}^\infty a_{2n-m}e_{n} , \sum_{j=0}^\infty a_{2j-m}e_{j}\Big\rangle + \sum_{m=0}^\infty\Big\langle \sum_{n=0}^\infty a_{2n+m+1}e_{n} , \sum_{j=0}^\infty a_{2j+m+1}e_{j}\Big\rangle &\\
&= \sum_{m=0}^\infty \Big(\sum_{n=0}^\infty |a_{2n-m}|^2\Big) + \sum_{m=0}^\infty \Big(\sum_{n=0}^\infty |a_{2n+m+1}|^2\Big). 
\end{flalign*}
Since the operator $V_\phi$ is Hilbert-Schmidt, therefore it follows that
$ \sum_{m=0}^\infty \|V_\phi e_m\|^2 = \sum_{m=0}^\infty  \left\langle V_\phi e_m, V_\phi e_m\right\rangle < \infty $. 
Hence, this implies that \[\sum_{m=0}^\infty \Big(\sum_{n=0}^\infty |a_{2n-m}|^2\Big) + \sum_{m=0}^\infty \Big(\sum_{n=0}^\infty |a_{2n+m+1}|^2\Big) < \infty \] and its if and only if $|a_n|=0 $ for each $n$. Thus, $\phi =0$.
\end{proof}
\noindent It is known that the only compact slant Toelitz operators on the Hardy space is the zero operator \cite{arora5}. Following theorem shows that the same holds true for slant H-Toeplitz operators.

\begin{thm}The slant H-Toeplitz operator $V_\phi$ is compact if and only if $\phi=0$.
\end{thm}
\begin{proof}
Let $\phi$ be a bounded measurable function and $V_\phi$ be a slant H-Toeplitz operator with the matrix $(a_{i,j})$ with respect to orthonormal basis $\{e_n\}_{n\geq 0}$ satisfying relations \eqref{matrix} and \eqref{mmatrix}.
Let $V_\phi$ be a compact operator. Since $e_n$ converges to $0$ weakly and therefore $\|V_\phi e_n\| \to 0$. This implies that $\|WT_\phi e_n\| \to 0$ and $\|WH_\phi e_n\| \to 0$. This further implies that $\left\langle \phi, e_n \right\rangle = 0$ for each $n\in \mathbb{Z}$. Hence, $\phi =0$, that is, $V_\phi =0$. Thus, the only compact slant H-Toeplitz operator is the zero operator.
\end{proof}
 The slant H- Toeplitz operator $V_\phi$ is a non-normal operator, that is, $V_\phi V_\phi^* \neq V_\phi^* V_\phi$. Moreover, in the following theorem we prove that zero operator is the only hyponormal slant H-Toeplitz operator.
\begin{thm} For $\phi \in L^\infty$, the operator $V_\phi$ is hyponormal if and only if $\phi=0$.
\end{thm}
\begin{proof}
For $\phi=0$,  the operator $V_\phi$ is trivially hyponormal. Now let the operator $V_\phi$ be hyponormal and $\phi= \sum_{n= -\infty}^\infty a_ne_n$. Then by definition of hyponormal it follows that
\begin{equation} \label{eqns2}
 \|{V_\phi}^*f\|^2 \leq \|V_\phi f\|^2 \ \text{ for all }  f\in H^2.
\end{equation}
In particular for $f(z)=e_0(z)$ in \eqref{eqns2}, we have $\ \|V_\phi ^* e_0\|^2 \leq \|V_\phi e_0\|^2$, 
that is, $\left\|K^*\left(\sum_{n=-\infty}^\infty \overline{a_{-n}}e_n \right)\right\|^2 \leq \left\|W\sum_{n=0}^\infty a_{n}e_n\right\|^2$ and this further implies that $\left\|\sum_{n=0}^\infty \overline{a_{-n}}e_{2n} + \sum_{n=0}^\infty \overline{a_{n+1}}e_{2n+1} \right\|^2 \leq \left\|\sum_{n=0}^\infty a_{2n}e_n \right\|^2.$ 
Therefore, on expanding we have $\sum_{n=0}^\infty |a_{-n}|^2$ + $\sum_{n=0}^\infty |a_{n+1}|^2$ $\leq$ $\sum_{n=0}^\infty |a_{2n}|^2$ which implies that $\sum_{n=1}^\infty |a_{-n}|^2 + \sum_{n=0}^\infty |a_{2n+1}|^2 \leq 0 $.
So, this gives that $a_{-n}=0 \ \text{for } \ n\geq 1$ and $a_{2n+1}= 0 \ \text {for } n\geq 0$. Similarly on taking $f(z)= e_1(z)=z$ in \eqref{eqns2}, it follows that
$ \|V_\phi ^* e_1\|^2 \leq \|V_\phi e_1\|^2$. Now from the definitions of operators $V_\phi$ and $V_\phi^*$ 
we have $\left\|\sum_{n=-\infty}^\infty \overline{a_{n+3}}e_{2n+1}\right\|^2 \leq \left\|\sum_{n=0}^\infty \overline{a_{2n+1}}e_{n} \right\|^2$, or equivalently,  $\sum_{n=-\infty}^\infty |a_{n+3}|^2 \leq \sum_{n=0}^\infty |a_{2n+1}|^2$, that is, $\sum_{n=-\infty}^\infty |a_{2n}|^2 \leq 0.$
Therefore, $|a_{2n}|=0$ for each $n$ and hence we have that $a_n=0$ for each $n\in \mathbb{Z}$. Thus, it follows that $\phi =0$.
\end{proof}
\noindent It is evident that every isometry operator is hyponormal, therefore it follows that a slant H-Toeplitz operator can not be isometry.

\section{Characterizations of slant H-Toeplitz operator}
\noindent Let $U \in \mathcal{B}(H^2)$ be a forward shift operator, that is, $U(f(z))= z(f(z))$ for all $f\in H^2$, $z\in \mathbb{T}$ and the operator $U^*$ denotes the adjoint of $U$. Let the operator $\mathcal{W} \in \mathcal{B}(L^2)$ be the multiplication operator with symbol $z$.
\begin{thm}\label{t2} If $A$ is a bounded linear operator on $H^2$ whose matrix with respect to orthonormal basis 
$\{e_n\}_{n\geq 0}$ is a slant H-Toeplitz matrix, then for each non-negative integers $m$, there exist a bounded linear operator $A_m$ defined from $H^2$ to $L^2$ which satisfies the following:
\begin{enumerate}
\item[(a)]$A_mC_{z^2}= {\mathcal{W}^*}^m AC_{z^2}U^{2m}$
\item[(b)]$U^*A_mM_{z^3}C_{z^4}= AM_{z^3}C_{z^4}U$
\item[(c)]$U^*A_m z^0= AM_{z^3}z^0.$
\end{enumerate}
\end{thm}
\begin{proof}
Let $A \in \mathbb{B}(H^2)$ has a slant H-Toeplitz matrix $(\alpha_{i,j})$  with respect to orthonormal basis $\{e_n\}_{n\geq 0}$. For each non-negative integers $m$, define a bounded linear operator $A_m$ on $H^2$ to $H^2\  \cup$ span$\{e_{-1},e_{-2},e_{-3}, \dots,e_{-m}\}$ $\subset L^2$ such that its matrix representation is given by
\begin{equation} \label{matrix1}
  \alpha_{k,0}= \begin{cases}
 \alpha_{k+j,4j}, & \text{for } j\geq 0 \text { and } k\geq -m \\ 
 \alpha_{k-j, 4j-1}, & \text{for } j=1,2,3,\dots \text{ and } k-j \geq -m
 \end{cases}
 \end{equation}and
 \begin{equation}\label{matrix1.1}  
  \alpha_{-m,2k}= a_{i, 2k+4i}, \quad \text{for } i \geq -m \text{ and } k\geq 1.
\end{equation}
\noindent Using the matrix representation of operators $A_m$, it follows that \begin{equation}\label{eq1}
\alpha_{q,2p}= \alpha_{q+j, 2p+4j}  \text { for }  p,q,j \geq 0.
\end{equation} 
Therefore, in particular for $j=m \geq 1$ by equation \eqref{eq1}, it follows that 
\begin{equation}\label{eq3}
\alpha_{q,2p}= \alpha_{q+m, 2p+4m}  \text{ for all } p,q \geq 0
\end{equation} 
Then, by equation $\eqref{eq3}$, it follows that $$\left\langle A_me_{2p}, e_q \right\rangle= \left\langle A e_{2p+2m}, e_{q+m} \right\rangle,$$ or equivalently, $$\left\langle A_mC_{z^2} e_p, e_q\right\rangle= \left\langle \mathcal{W}^*AC_{z^2}U^{2m}e_p, e_q \right\rangle$$  for all $p\geq 0,\  m \geq 1\text{ and } q \geq -m$ and hence $A_mC_{z^2}= {\mathcal{W}^*}^m AC_{z^2}U^{2m}$.\\
Since, $\alpha_{k,0}= \alpha_{k-j,4j-1}\ \text{for all } j\geq 1\text{ and } k-j\geq -m$, therefore it follows that $\alpha_{k+r+2,0}= \alpha_{k+r+2-j, 4j-1}\quad \text {for }j\geq 1, k\geq j-m$ and $r \geq 0.$ 
 In particular for $j=r+1, r+2$, we see that \[\alpha_{k+1,4r+3}= \alpha_{k, 4r+7}\  \text{for all } r\geq 0, k\geq 0.\] Therefore, $\left\langle A_m e_{4r+3}, e_{k+1} \right\rangle = \left\langle A e_{4r+7}, e_k \right\rangle$, or equivalently, we get that $\left\langle U^*A_mM_{z^3}C_{z^4} e_r,e_k \right\rangle= \left\langle AM_{z^3}C_{z^4}U e_r, e_k\right\rangle$ for each $r,k\geq 0$
 and hence it follows that $U^*A_mM_{z^3}C_{z^4}=AM_{z^3}C_{z^4}U$.
 Again by the definition of matrix $(\alpha_{i,j})$, it follows that $\alpha_{k+1,0}= \alpha_{k,3}$ for all $ k \geq 0$, which implies that $\left\langle A_m e_0, e_{k+1}\right\rangle= \left\langle Ae_3, e_k \right\rangle$, or equivalently, $\left\langle U^* A_m e_0, e_k \right\rangle= \left\langle AM_{z^3}e_0, e_k\right\rangle$. Thus, it gives $U^*A_m z^0= AM_{z^3}z^0.$
\end{proof}

\begin{lem} For all $j \geq 0 \text{ and } i \geq -m, \ \left\langle A_m e_j, e_i \right\rangle$ is independent of $m$. 
\end{lem}
\begin{proof}
Let $(\alpha_{i,j})$ be the matrix of $A_m$ with respect to orthonormal basis $\{e_n \}_{n \geq 0}$ satisfying equations \eqref{matrix1} and \eqref{matrix1.1}. For $0>i>-m$, using the matrix definition of operator $A_m$, we get that
\[ \left\langle A_me_{4p}, e_i \right\rangle = \alpha_{0, 4p+4k} = \left\langle Ae_{4(p+k)}, e_0 \right\rangle \text{ where }i+k=0\] and this is true for each non-negative integer $p$. Also, for $p\geq 1$, we have
$\left\langle A_m e_{4p-1} ,e_i\right\rangle = \alpha_{i+p,0}= \left\langle Ae_0, e_{i+p} \right\rangle.$ 
For non-negative integers $i$ and $j$ and by using the relation $A_mC_{z^2}= {\mathcal{W}^*}^m AC_{z^2}U^{2m}$, it follows that $
\left\langle A_mC_{z^2}e_j, e_i \right\rangle = \left\langle {\mathcal{W}^*}^m AC_{z^2}U^{2m}e_j, e_i \right\rangle =\left\langle AC_{z^2} U^{2m}, U^m e_i \right\rangle 
=\left\langle {U^*}^m AC_{z^2} U^{2m}e_j, e_i\right\rangle 
= \left\langle AC_{z^2}e_j, e_i \right\rangle.$
Also, the relation ${U^*} A_mM_{z^3}C_{z^4}= AM_{z^3}C_{z^4}U$ implies that
$\left\langle {U^*} A_mM_{z^3}C_{z^4} e_j, e_i \right\rangle = \left\langle AM_{z^3}C_{z^4}Ue_j,e_i  \right\rangle $ which gives that  $\left\langle A e_{4j+7}, e_i \right\rangle 
= \left\langle Ae_{4j+3}, e_{i+1} \right\rangle$. Moreover, from the definition of $A_m$, we have $\left\langle {U^*} A_mM_{z^3}C_{z^4} e_j, e_i  \right\rangle = \left\langle A_m e_{4j+3}, e_{i+m} \right\rangle = \left\langle Ae_{4j+3}, e_{i+1} \right\rangle.$
Again the relation $U^*A_m z^0= AM_{z^3}z^0$ implies that $ \left\langle U^*A_m e_0, e_i \right\rangle = \left\langle AM_{z^3} e_0, e_i \right\rangle= \alpha_{i,3}= \alpha_{i+1,0} =  \left\langle Ae_0, e_{i+1}\right\rangle$
and also $\left\langle U^*A_m e_0, e_i \right\rangle= \left\langle A_me_0, e_{i+1} \right\rangle$. Therefore, for all non-negative integers $i \text{ and }j $ we get that $ \left\langle A_m e_j, e_i \right\rangle =  \left\langle A e_j, e_i\right\rangle$. Hence, for all $j \geq 0 \text{ and }i \geq -m, \  \left\langle A_m e_j, e_i \right\rangle$ is independent of $m$.
\end{proof}
\begin{exmp}Let $A$ be a bounded linear operator on $H^2$ whose matrix $(a_{i,j})$ with respect to orthonormal basis $\{e_n\}_{n\geq 0}$ is a slant H-Toeplitz matrix and satisfies the relations \eqref{matrix} and \eqref{mmatrix}.
Then the matrices of the operators $A_1$ and $A_2$ defined in the Theorem \ref{t2}, are given by
\[A_1= \begin{bmatrix}
a_{-2} & a_{-1} & a_{-3} & a_{0} & a_{-4} & a_{1} & a_{-5}\cdots\\
a_0 & a_1 & a_{-1} & a_2 & a_{-2} & a_{3} & a_{-3}\cdots\\
a_2 & a_3 & a_{1} & a_4 & a_{0} & a_{5} & a_{-1}\cdots\\
a_4 & a_5 & a_{3} & a_6 & a_{2} & a_{7} & a_{1}\cdots\\
a_6 & a_7 & a_{5} & a_8 & a_{4} & a_{9} & a_{3}\cdots\\
a_8 & a_9 & a_{7} & a_{10} & a_{6} & a_{11} & a_{5}\cdots\\
\vdots & \vdots & \vdots & \vdots & \vdots & \vdots & \vdots
\end{bmatrix}\]
and \[A_2= \begin{bmatrix}
a_{-4} & a_{-3} & a_{-5} & a_{-2} & a_{-6} & a_{-1} & a_{-7}\cdots\\
a_{-2} & a_{-1} & a_{-3} & a_{0} & a_{-4} & a_{1} & a_{-5}\cdots\\
a_0 & a_1 & a_{-1} & a_2 & a_{-2} & a_{3} & a_{-3}\cdots\\
a_2 & a_3 & a_{1} & a_4 & a_{0} & a_{5} & a_{-1}\cdots\\
a_4 & a_5 & a_{3} & a_6 & a_{2} & a_{7} & a_{1}\cdots\\
a_6 & a_7 & a_{5} & a_8 & a_{4} & a_{9} & a_{3}\cdots\\
a_8 & a_9 & a_{7} & a_{10} & a_{6} & a_{11} & a_{5}\cdots\\
\vdots & \vdots & \vdots & \vdots & \vdots & \vdots & \vdots
\end{bmatrix} \]
Similarly, we can obtain the matrix representation for the other operators $A_m$ for $m>2$ and clearly all these operators satisfying the conditions given in Theorem \ref{t2}.
\end{exmp}
In the following theorem we give the characterization for slant H-Toeplitz operators.
\begin{thm} A necessary and sufficient condition for an operator on $H^2$ to be a slant H-Toeplitz operator is that its matrix with respect to the orthonormal basis $\{e_n\}_{n\geq 0}$ is the slant H-Toeplitz matrix.
\end{thm}
\begin{proof}
Every slant H-Toeplitz operator defined on $H^2$ has a  slant H-Toeplitz matrix with respect to the orthonormal basis $\{e_n\}_{n\geq 0}$. Conversely, let us assume that $A$ be a bounded linear operator on $H^2$ whose matrix with respect to the orthonormal basis $\{e_n\}_{n \geq 0}$ is a slant H-Toeplitz matrix. So, we claim that $A$ is a slant H-Toeplitz operator.\\
For each non-negative integers $m$, consider a bounded linear operator $A_m$ defined on $H^2$ to $L^2$ such that its matrix satisfies the conditions \eqref{matrix1} and \eqref{matrix1.1} given in the Theorem \ref{t2}. Then, the operators $A_m$ satisfies the following 
\begin{enumerate}
\item[(a)]$A_mC_{z^2}= {\mathcal{W}^*}^m AC_{z^2}U^{2m}$
\item[(b)]$U^*A_mM_{z^3}C_{z^4}= AM_{z^3}C_{z^4}U$
\item[(c)]$U^*A_m z^0= AM_{z^3}z^0.$
\end{enumerate}
Moreover, for non-negative integers $i$ and $j$, we have $\left\langle A_m z^j, z^i\right\rangle = \left\langle Az^j, z^i \right\rangle $. If $p \text{ and }q$ are finite linear combinations of $z^i$ for $i\geq 0$, then the sequence $\{\left\langle A_m p, q \right\rangle\}$ is convergent. Therefore, the sequence $\{A_m\}$ of operator on $H^2$ is weakly convergent to a bounded linear operator say, $B$, on $H^2$. Then, for $i,j \geq 0$, it follows 
\begin{flalign*}
\left\langle PB z^j, z^i\right\rangle & = \left\langle B z^j, z^i\right\rangle = \lim _{m \to \infty} \left\langle A_m z^j, z^i\right\rangle = \left\langle A z^j, z^i \right\rangle=  a_{i,j}  &
\end{flalign*} 
and if $f$ and $g$ are in $H^2$, then we have
\begin{flalign*}
\left\langle P Bf, g\right\rangle & = \lim _{m \to \infty} \left\langle A_m f, g\right\rangle = \lim _{m \to \infty} \left\langle A f, g \right\rangle  = \left\langle A f, g \right\rangle.
\end{flalign*} 
Therefore, $PB f = Af$ for each  $f \in H^2$. Hence, it follows that operator $A$ is a slant H-Toeplitz operator on $H^2$. Let the Fourier coefficients of $\phi$ that induces the operator $A$ from its matrix are given by \begin{equation*}
\left\langle \phi, z^{k} \right\rangle = \begin{cases}
a_{k/2,0}, &  k\geq 0 \text{ and $k$ is even} \\
a_{(k-1)/2,1} & k\geq 0 \text{ and $k$ is odd}\\
a_{0,-2k}, &  k\leq -1. 
\end{cases}
\end{equation*}
Let if $f(z)= z^n \in H^2$ then $AC_{z^2} f(z)= Az^{2n}$ and $AM_zC_{z^2}= Az^{2n+1}$. Since the operators $AC_{z^2}$ and $AM_zC_{z^2}$ are slant Toeplitz and slant Hankel operators, respectively,  $AC_{z^2}= WT_\phi$ and $AM_zC_{z^2}= WH_\phi.$ Then for each function $f_1(z^2)\in H^2$, we have
\[AC_{z^2}(f_1(z))= WT_\phi f_1(z)= WPM_\phi K (f_1(z^2))= V_\phi (f_1(z^2))\]  and for each $f_2(z^2)\in H^2$, we obtain
\begin{flalign*}
&AM_zC_{z^2} (f_2(z))= WH_\phi (f_2(z))= WPM_\phi J(f_2(z))= WPM_\phi(z^{-1}f_2(z^{-1}))&\\
&= WPM_\phi K(zf_2(z^2))= V_\phi(zf_2(z^2)).
\end{flalign*}
If $h(z)\in H^2$, then $h(z)= h_1(z^2)+ zh_2(z^2)$. Moreover, we have 
\begin{flalign*}
A(h(z))&= A(h_1(z^2)+ zh_2(z^2)) = A(h_1(z^2))+ A(zh_2(z^2))&\\
&= AC_{z^2}(h_1(z)) + AM_zC_{z^2}(h_2(z)) &\\
&= V_\phi((h_1(z^2))+ V_\phi((zh_2(z^2))&\\
&= V_\phi(h_1(z^2)+ zh_2(z^2))= V_\phi (h(z))
\end{flalign*}
which is true for every $h(z) \in H^2$. Hence, the operator $A$ is a slant H-Toeplitz operator with symbol $\phi$.
\end{proof}

In the following theorem we give another characterization for slant H-Toeplitz operators.

\begin{thm}\label{chract2} A bounded linear operator $A$ on $H^2$ is a slant H-Toeplitz operator if and only if satisfies \begin{enumerate}
\item[(a)]$AC_{z^2}= {U^*} AC_{z^2}U^{2}$
\item[(b)]$U^*AM_{z^3}C_{z^4}= AM_{z^3}C_{z^4}U$
\item[(c)]$U^*A z^0= AM_{z^3}z^0.$
\end{enumerate}
\end{thm}
\begin{proof}
Let the operator $A \in \mathcal{B}(H^2)$  satisfies the conditions (a),  (b) and (c).
Then from (a) and (b) it follows that $AC_{z^2}$ is a slant Toeplitz operator and $AM_zC_{z^2}$ is a slant Hankel operator. Also if $f\in H^2$, then 
\[{U^*} AC_{z^2}U^2 (f(z))= AC_{z^2}(f(z))\] and
\[U^*AM_{z^3}C_{z^4}(f(z)) = AM_{z^3}C_{z^4}U(f(z)).\] This gives that, 
\begin{equation}\label{eqn1}
U^*A(z^4f(z^2))= A(f(z^2))\text{ and } U^*A(z^3f(z^4))= A(z^7f(z^4)).
\end{equation}
This is true for each functions $f(z^2),f(z^4)\in H^2$. Therefore, in particular for $f(z^2)=z^0,z^2,z^4,z^6, \dots$ and $f(z^4)= z^0,z^4,z^8,z^{12}, \dots$  using equation \eqref{eqn1}, we obtain the following relations:
\begin{flalign}\label{eqr1}
U^*A(z^{2n+4}) = A(z^{2n}) \text{ and } U^*A(z^{4n+3})= A(z^{4n+7})\text{
 for } n\geq 0. 
\end{flalign} 
Let $(a_{i,j})$ be the matrix of the bounded linear operator $A$ with respect to orthonormal basis $\{e_n\}_{n\geq 0}$. Then, for all $k \geq 0$ and by the relation \eqref{eqr1}, we have  
\begin{flalign*}
a_{k,0} &= \left\langle Az^0, z^k\right\rangle = \left\langle U^*A z^4, z^k \right\rangle = \left\langle Az^4, z^{k+1}\right\rangle = a_{k+1,4}&\\
&= \left\langle U^*A z^8, z^{k+1} \right\rangle = \left\langle Az^8, z^{k+2}\right\rangle = a_{k+2,8} &\\
&= \left\langle U^*A z^{12}, z^{k+2} \right\rangle = \left\langle A z^{12}, z^{k+3} \right\rangle= a_{k+3,12}
\end{flalign*}
and so on. On continuing in this manner, for $j \geq 1$ and $ k\geq 0$, we obtain that $a_{k,0}= a_{k+j, 4j}$. Again for each $k\geq 1$ and by the relation \eqref{eqr1}, it follows that \begin{flalign*}
a_{k,0} &= \left\langle Az^0, z^k\right\rangle = \left\langle U^*A z^0, z^{k-1} \right\rangle = \left\langle AM_{z^3}z^0, z^{k-1}\right\rangle =  a_{k-1,3}&\\
&= \left\langle U^*A z^3, z^{k-2} \right\rangle = \left\langle Az^7, z^{k-2}\right\rangle = a_{k-2,7} &\\
&= \left\langle U^*A z^{7}, z^{k-3} \right\rangle = \left\langle A z^{11}, z^{k-3} \right\rangle= a_{k-3,11} 
\end{flalign*}
 and so on. Again on continuing the same manner it follows that $\text{for all } k\geq 1$, 
$a_{k,0}= a_{k-1,3}= a_{k-2,7}= a_{k-3,11} \dots \dots = a_{0,4k-1}$.\linebreak   
Therefore, for each $k\geq 1 \text { and for } j=1,2,3,\dots , k-j \geq0$, it follows that $\ a_{k,0}= a_{k-j,4j-1}.$
Again for $k\geq 0$ and from the relation \eqref{eqr1}, it follows that \begin{flalign*}
a_{0,2k} &= \left\langle Az^{2k}, z^0\right\rangle = \left\langle U^*A z^{2k+4}, z^0 \right\rangle = \left\langle Az^{2k+4}, z^{1}\right\rangle =  a_{1,2k+4}&\\
&= \left\langle U^*A z^{2k+8}, z^{1} \right\rangle = \left\langle Az^{2k+8}, z^{2}\right\rangle = a_{2,2k+8} &\\
&= \left\langle U^*A z^{2k+12}, z^{2} \right\rangle = \left\langle A z^{2k+12}, z^{3} \right\rangle= a_{3,2k+12}  
\end{flalign*}
and so on.
Therefore, on continuing the same process, for all $k\geq 0$ and  $i \geq 1$ it follows that
  $a_{0,2k}= a_{i, 2k+4i}$. Since the matrix $(a_{i,j})$ satisfies the relations \eqref{matrix} and \eqref{mmatrix}, therefore the  matrix  $(a_{i,j})$ is  a slant H-Toeplitz matrix. Thus, the operator $A$ is a slant H-Toeplitz operator on $H^2$ with symbol $\phi$ whose Fourier coefficients are given by 
\begin{equation*}
\left\langle \phi, z^{k} \right\rangle = \begin{cases}
a_{k/2,0}, &  k\geq 0 \text{ and $k$ is even} \\
a_{(k-1)/2,1} & k\geq 0 \text{ and $k$ is odd}\\
a_{0,-2k}, &  k\leq -1. 
\end{cases}
\end{equation*} 
Conversely, let the operator $A$ be a slant H-Toeplitz operator on $H^2$. Then, $A=V_\phi$ for some non-zero $\phi \in L^\infty$. So, for each $f \in H^2$ we have
\[ AC_{z^2}(f(z))= V_\phi C_{z^2}(f(z))= WPM_\phi K (f(z^2))= WT_\phi (f(z)).
\] Hence, $AC_{z^2}$ is a slant Toeplitz operator and therefore we get that,  $U^*AC_{z^2}U^2= AC_{z^2}$.
Also for each $f \in H^2$, it follows
\begin{flalign*}
AM_zC_{z^2}(f(z))&=  V_\phi(zf(z^2))=WPM_\phi K(zf(z^2))= WPM_\phi (z^{-1}f(z^{-1}))&\\
&= WPM_\phi J(f(z))= WH_\phi (f(z)).
\end{flalign*}
Therefore, the operator $AM_zC_{z^2}$ is a slant Hankel operator and hence 
$U^*AM_{z^3}C_{z^4}= AM_{z^3}C_{z^4}U$. 
Again if $\phi(z)= \sum\limits_{n=-\infty}^\infty a_n z^n$, then the operator $A$ satisfies the following:\[U^*A(z^0)= U^*V_\phi(z^0)= U^*WP\phi(z)= U^*\sum_{n=0} ^\infty a_{2n}z^n= \sum_{n=0} ^\infty a_{2n+2}z^n\] and 
\[AM_{z^3}(z^0)= V_\phi(z^3)= WPM_\phi K(z^3)=WP\phi(z^{-2})= \sum_{n=0} ^\infty a_{2n+2}z^n.\] Therefore,  $AM_{z^3}(z^0)= U^*A(z^0).$ Thus, every slant H-Toeplitz operator satisfies the above three conditions of the theorem.
\end{proof}

\noindent In the following theorem, we have shown that there does not exist any  non-zero self-adjoint slant H-Toeplitz operator on $H^2$.
\begin{thm}The slant H-Toeplitz operator $V_\phi$ with the symbol $\phi$ is self adjoint if and only if $\phi=0$.
\end{thm}
\begin{proof}
If $\phi=0$, then result is obvious. Now suppose that the operator $V_\phi \neq 0$ and its adjoint $V_\phi^*$ is a slant H-Toeplitz operator. Then, by Theorem \eqref{chract2}, the operator $V_\phi^*$ satisfies the following- 
\begin{enumerate}
\item[(a)]${U^*} V_\phi^* C_{z^2}U^2= V_\phi^* C_{z^2}$
\item[(b)]$U^*V_\phi^* M_{z^3}C_{z^4}= V_\phi^* M_{z^3}C_{z^4}U$
\item[(c)]$U^*V_\phi^* z^0= V_\phi^* M_{z^3}z^0.$
\end{enumerate}
For $\phi(z)= \sum_{n=-\infty}^\infty a_n z^n$, the relation (c) implies that $U^*(K^*M_{\bar{\phi}})(1) = K^*M_{\bar{\phi}}W^*(z^3)$, or equivalently, $U^* K^* \Big(\sum\limits _{n=-\infty}^\infty \overline{a_n}\bar{z}^n \Big)= K^* \Big( \sum\limits _{n=-\infty}^\infty \overline{a_n}\bar{z}^{n+6} \Big)$ which gives that 
\begin{flalign*}
&U^* \Big(\sum_{n=0}^\infty \overline{a_{-n}}z^{2n} +  \sum_{n=0}^\infty \overline{a_{n+1}}z^{2n+1}\Big) = K^*\Big(\sum_{n=0}^\infty \overline{a_{-n-6}}z^{n} + \sum_{n=0}^\infty \overline{a_{n-5}}z^{-n-1} \Big)&\\
&\text{and this gives } \sum_{n=1}^\infty \overline{a_{-n}}z^{2n} +  \sum_{n=1}^\infty \overline{a_{n+1}}z^{2n+1} =
\sum_{n=0}^\infty \overline{a_{-n-6}}z^{2n} + \sum_{n=0}^\infty \overline{a_{n-5}}z^{2n+1}. 
\end{flalign*}
Therefore, on comparing the coefficients we get $\overline{a_{-6}}=0, \overline{a_{-5}}=0$ and $\overline{a_{-n}}= \overline{a_{-n-6}},\  \overline{a_{n+1}}= \overline{a_{n-5}}\ $ for $n\geq 1$.
Now since $a_n \to 0 \text{ as } n \to \infty$, therefore this implies that $a_n=0 $ for each $n$ and hence  $\phi=0$.
\end{proof}

\noindent Next we show that a non-zero slant Toeplitz operator can not be a slant H-Toeplitz operator.
\begin{thm}A slant Toeplitz operator $B_\phi$ is a slant H-Toeplitz operator if and only if $\phi=0$.
\end{thm}
\begin{proof}
If $\phi=0$, then result is trivial. Let $B_\phi$ be a slant H-Toeplitz operator. Then, by using the Theorem \eqref{chract2}, the operator $B_\phi$ satisfies the following:
\begin{enumerate}
\item[(a)]${U^*} B_\phi C_{z^2}U^2= B_\phi C_{z^2}$
\item[(b)]$U^*B_\phi M_{z^3}C_{z^4}= B_\phi M_{z^3}C_{z^4}U$
\item[(c)]$U^*B_\phi z^0= B_\phi M_{z^3}z^0.$
\end{enumerate}
Taking $\phi(z)=\sum_{n=-\infty}^\infty a_nz^n$ and as
 ${U^*} B_\phi C_{z^2}U^2(z^m)= {U^*} B_\phi(z^{2m+4})$. Then by using part (a) we get $\left\langle U^*B_\phi C_{z^2}U^2 z^m, z^j\right\rangle = \left\langle B_\phi C_{z^2} z^m, z^j \right\rangle$ which implies that $\left\langle B_\phi z^{2m+4}, z^{j+1} \right\rangle = \left\langle B_\phi z^{2m}, z^j \right\rangle$. This gives that
\begin{flalign*}  
& \Big\langle WP\Big(\sum_{n=-\infty}^\infty a_nz^{2m+n+4}\Big), z^{j+1}\Big\rangle =\Big\langle WP\Big(\sum_{n= -\infty}^\infty a_nz^{n+2m}\Big), z^j\Big\rangle &\\
& \text{or, equivalently, } \Big\langle \sum_{n=-2m-4}^\infty a_nz^{2m+n+4}, z^{2j+2}\Big\rangle =\Big\langle \sum_{n=-2m}^\infty a_nz^{n+2m}, z^{2j} \Big\rangle 
\end{flalign*}
and therefore, $a_{2j-2m-2}= a_{2j-2m} \text{ for all } m,j\geq 0$. Now on substituting $m,j=0,1,2,3,\dots$, we get that $a_0=a_{2n}$ for all integer $n$. Since $a_n \to 0 \text { as }n \to \infty$, therefore for each integer $n$, we get that $a_{2n}= 0$. Now for $m\geq 0$, we have $U^*B_\phi M_{z^3}C_{z^4}(z^m)= U^*B_\phi M_{z^3}(z^{4m})= U^* B_\phi (z^{4m+3})$ and  $B_\phi M_{z^3} C_{z^4} U(z^m)= B_\phi M_{z^3}(z^{4m+4})= B_\phi (z^{4m+7})$.
Then, from the relation (b) it follows $\left\langle U^*B_\phi M_{z^3}C_{z^4} z^m, z^j\right\rangle = \left\langle B_\phi M_{z^3} C_{z^4} Uz^m, z^j \right\rangle$, that is, $\left\langle U^* B_\phi z^{4m+3}, z^j \right\rangle = \left\langle B_\phi z^{4m+7}, z^j \right\rangle$. This further implies that
\begin{flalign*}
& \Big\langle WP\Big(\sum_{n=-\infty}^\infty a_nz^{4m+n+3}\Big), z^{j+1}\Big\rangle =\Big\langle WP\Big(\sum_{n= -\infty}^\infty a_n z^{n+4m+7}\Big), z^j\Big\rangle \text{ or, }&\\
& \text{ equivalently, } \Big\langle \sum_{n=-4m-3}^\infty a_n z^{n+4m+3}, z^{2j+2}\Big\rangle =\Big\langle \sum_{n=-4m-7}^\infty a_n z^{n+4m+7}, z^{2j} \Big\rangle.
\end{flalign*}
Thus, it gives that $a_{2j-4m-1}= a_{2j-4m-7} \text{ for all } j,m\geq 0$ and this implies that $a_1= a_{2n+1}\ \text{for all integers } n $. Since $a_n \to 0 \text { as }n \to \infty$, therefore it follows that $a_{2n+1}= 0  \text{ for all integers } n $ and hence $\phi= 0$.
\end{proof}
\begin{thm}If a slant Hankel operator $L_\phi$ is a slant H-Toeplitz operator, then $\phi \in (z+z^3H^\infty)^\perp$, where $(z+z^3H^\infty)= \{z+z^3 \psi: \ \psi \in H^\infty\}.$
\end{thm}
\begin{proof}
Let the operator $L_\phi= WH_\phi$ be a slant H-Toeplitz operator. Then, by the Theorem \eqref{chract2}, the operator $B_\phi$ satisfies the following:
\begin{enumerate}
\item[(a)]${U^*} WH_\phi C_{z^2}U^2= WH_\phi C_{z^2}$
\item[(b)]$U^*WH_\phi M_{z^3}C_{z^4}= WH_\phi M_{z^3}C_{z^4}U$
\item[(c)]$U^*WH_\phi z^0= WH_\phi M_{z^3}z^0.$
\end{enumerate}
Taking $\phi(z)= \sum_{n= -\infty}^\infty a_nz^n \in L^\infty$ and by using the relation (a), we obtain
$\left\langle {U^*} WH_\phi C_{z^2}U^2 z^m, z^j \right\rangle = \left\langle WH_\phi C_{z^2} z^m, z^j \right\rangle$   which implies that 
$\left\langle H_\phi z^{2(m+2)}, z^{2(j+1)}\right\rangle = \left\langle H_\phi z^{2m}, z^{2j}\right\rangle.$  So, from the matrix representation 
of the operator $H_\phi$ it follows  that 
\begin{flalign} 
& \qquad \label{m1} a_{2m+2j+7} = a_{2m+2j+1}  \text { for all } m,j\geq 0.&
\end{flalign}
Again on using the relation (b), it follows that
\begin{flalign}
&\left\langle {U^*} WH_\phi M_{z^3}C_{z^4} z^m, z^j \right\rangle = \left\langle WH_\phi M_{z^3}C_{z^4}U z^m, z^j \right\rangle \text{ which implies that }&\nonumber\\
& \left\langle H_\phi z^{4m+3}, z^{2(j+1)}\right\rangle = \left\langle H_\phi z^{4m+7}, z^{2j}\right\rangle. \text{ Using matrix representation of  } &\nonumber\\
&\text{ the operator $H_\phi$ the above condition is equivalent to following: } \nonumber \\
& \qquad \label{m2} a_{4m+2j+6} = a_{4m+2j+8}   \text{ for all } m,j\geq 0.
\end{flalign}
Moreover, from the relation (c), it follows that $\left\langle W H_\phi M_{z^3}z^0, z^j\right\rangle = \left\langle U^*WH_\phi z^0, z^j\right\rangle$ and then $\left\langle H_\phi z^3, z^{2j}\right\rangle = \left\langle H_\phi z^0, z^{2(j+1)}\right\rangle.$ Therefore, using the matrix representation of the operator $H_\phi$ we obtain following relation \begin{flalign} 
\label{m3}&\qquad a_{2j+4}= a_{2j+3} \text{ for all } j\geq 0.&
\end{flalign}
On substituting $m,j=0,1,2,\dots$ in equations \eqref{m1}, \eqref{m2}, \eqref{m3}, we obtain \[
a_{2k-1}= a_{2k+5},\  a_{2k+1}=a_{2k+2} \text{ and } a_{2k+4}= a_{2k+6}, \quad  k\in \mathbb{N}.\]
This implies that $a_{1}= a_{n}$ for each $n\geq 3$. Since $a_n \to 0 \text { as }n \to \infty$, we get that $\phi(z)= \sum_{n=-\infty}^{0}a_nz^n+ a_2z^2$. Hence,  $\phi \in (z+z^3H^\infty)^\perp$.
\end{proof}\vspace{0.4cm}

 We can extend the notion of slant H-Toeplitz operator to the space $L^2$ by defining the operator $\breve{V_\phi}: L^2 \longrightarrow L^2$ such that $\breve{V_\phi}= WM_\phi K$, where $K: L^2 \longrightarrow L^2$ defined as $K_{e_{2n}}= e_n,\ K e_{2n+1}= e_{-n-1}$ and $W: L^2 \longrightarrow L^2$ as $W e_{2n}= e_n,\ W e_{2n+1}=0$ for each integer $n$. The same techniques  can be applied to prove the results for   $\breve{V_\phi}$.

 The notion of slant H-Toeplitz operator on $H^2$ can be further extended to generalised slant H-Toeplitz operators, which can be defined as the operator  $V_\phi^k \in \mathcal{B}(H^2)$ with the symbol $\phi \in L^\infty$ by $V_\phi^k(f)= W_kPM_\phi K(f)$ for each $f$ in $H^2$, where the operator $W_k \in \mathcal{B}(L^2)$ such that $W_k e_{2n}= e_n,\ W_k e_{2n+1}=0$ for each integer $n$. Clearly, for $k=2$, the operator $V_\phi^k$ is same as the slant H-Toeplitz operator $V_\phi$. Moreover, results for operator $V_\phi$ can be extended for the operator $V_\phi ^k$. 
\section*{Acknowledgement}
\noindent  Support of CSIR Research Grant to second author [F.No. 09/045(1405)\\
/2015-EMR-\textrm{I}] for carrying out the research work is fully acknowledged.

\bibliographystyle{amsplain}

\end{document}